\documentclass[a4paper,12pt]{amsart}
\usepackage{amsfonts}
\usepackage{amsmath}
\usepackage{amscd,amssymb,amsthm}
\numberwithin{equation}{section}
\swapnumbers
\theoremstyle{plain}
\newtheorem{theorem}[equation]{Theorem}
\newtheorem{lemma}[equation]{Lemma}

\newtheorem{corollary}[equation]{Corollary}

\newtheorem{remark}[equation]{Remark}

\newtheorem{openproblem}[equation]{Open problem}

\newcommand{\K}{{\mathcal K}}
\newcommand{\E}{{\mathcal E}}

\title{A Power Mean Inequality involving the complete elliptic integrals }
\author[G.D. Wang]{Gendi Wang}
\author[X.H. Zhang]{Xiaohui Zhang}
\author[Y.M. Chu]{Yuming Chu}
\address{College of Mathematics and Econometrics, Hunan University, Changsha 410082, China}
\email{wgdwang@gmail.com}
\address{Department of Mathematics and Statistics, University of Turku, FI-20014 Turku, Finland}
\email{xiazha@utu.fi}
\address{Department of Mathematics, Huzhou Teachers College, Huzhou 313000, China}
\email{chuyuming@hutc.zj.cn}
     
\date{}

\begin{document}

\maketitle

\begin{abstract}
In this paper the authors investigate a power
mean inequality for a special function which is defined by the complete elliptic integrals.
\end{abstract}



\section{Introduction}

Throughout this paper we let $r'=\sqrt{1-r^2}$ for $0<r<1$.
The well-known complete elliptic integrals of the first and second kind \cite{Bo, BF}
are respectively defined by
\begin{equation}\label{eqn:K}
\left\{\begin{array}{ll}
\K(r)=\int_{0}^{\pi/2}\frac{d \theta}{\sqrt{1-r^2\sin^2\theta}},\\
\K'(r)=\K(r'),\\
\K(0)=\frac{\pi}2,\qquad\K(1)=\infty,
\end{array}\right.
\end{equation}
and
\begin{equation}\label{eqn:E}
\left\{\begin{array}{ll}
\E(r)=\int_{0}^{\pi/2}\sqrt{1-r^2\sin^2\theta} d\theta,\\
\E'(r)=\E(r'),\\
\E(0)=\frac{\pi}2,\qquad\E(1)=1.
\end{array}\right.
\end{equation}
In the sequel, we will use the symbols $\K$ and
$\E$ for $\K(r)$ and $\E(r)$, respectively.
The complete elliptic integrals play a very important role in the study of conformal invariants \cite{AVV3}, quasiconformal
mappings \cite{AVV1, AVV2, AVV3, H} and Ramanujan's modular equations \cite{AQVV}. Numerous sharp inequalities and elementary approximations for the complete elliptic integrals have been proved in \cite{AQ, AQVa, AVV3, BPR, KN1, KN2}.

The special function $m(r)$ is defined as
$$m(r)=\frac2{\pi}r'^2\K(r)\K'(r),\quad\, 0<r<1.$$
This function is of importance in the research of distortion theory of quasiconformal mappings in the plane.  
Recently, various interesting inequalities of $m(r)$ have been obtained by several authors, see \cite{AQVa, AVV1, AVV2, AVV3, QVu}.

The power mean is
defined for $x,y>0$ and real parameter $\lambda$ by
$$M_{\lambda}(x,y)=\left(\frac{x^{\lambda}+y^{\lambda}}{2}\right)^{1/{\lambda}} \, \mbox{for}\, {{\lambda}\neq 0},\, \mbox{and}\, M_0(x,y)=\sqrt{xy}.$$
It is well known that $M_{\lambda}(x,y)$ is continuous and increasing with respect to $\lambda$.
Many interesting properties of power means are given in \cite{Bu}. Power mean inequalities for some special functions can be found in \cite{A, AVV3, R, WZJ, ZWC}.

In this paper, we shall show a power mean inequality for the special function $m(r)$.
Our main result is the following theorem:

\begin{theorem}\label{thm:pm4m} 
Let ${\lambda}$ be a real number.
The inequality
\begin{equation}\label{ineq:pm4m}
M_{\lambda}(m(x),m(y))\leq m(M_{\lambda}(x,y))
\end{equation}
holds for all $x,y\in (0,1)$ if and only if
${\lambda}\leq0$. The reverse of (\ref{ineq:pm4m}) holds for all $x,y\in (0,1)$ if and only if
${\lambda}\geq C>0$, where $C$ is some constant. The sign of equality is
valid in (\ref{ineq:pm4m}) if and only if $x=y$. 
\end{theorem}

\section{Lemmas}

In order to prove our main result we need some lemmas, which we
present in this section. We establish some properties of certain
functions, which are defined in terms of the complete elliptic
integrals of the first and second kinds, $\K$ and $\E$,
respectively.

Now we list some derivative formulas \cite[Appendix E, pp. 474-475]{AVV3}:
$$\frac{d\K}{dr}=\frac{\E-r'^2\K}{rr'^2},\quad \frac{d\E}{dr}=\frac{\E-\K}{r},$$
and
$$\frac{d}{dr}m(r)=\frac{\pi-4\E'\K}{\pi{r}},$$
where $0<r<1$. By the derivative of $m(r)$, it is easy to see that $m(r)$ is strictly decreasing from $(0,1)$ onto $(0,\infty)$.

The following Lemma \ref{lem:citedlem1} is from \cite[Lemma 5.2(2)]{AQVV} and  \cite[Theorem 3.21 (1), (7), and Exercises 3.43 (32)]{AVV3}.

\medskip

\begin{lemma}\label{lem:citedlem1}
\begin{enumerate}
\item The function $f_1(r)=r'^2\K/{\E}$ is decreasing from
$(0,1)$ onto $(0,1)$. 
\item The function $f_2(r)=(\E-r'^2\K)/{r^2}$ is strictly increasing and convex from $(0,1)$ onto $(\pi/4,1)$. 
\item For each $c\in [1/2,\infty)$, the function $f_3(r)=r'^c\K$ is decreasing from
$[0,1)$ onto $(0,\pi/2]$. 
\item The function $f_4(r)={r^{-2}(\K-\E)}/{\K}$ is increasing and convex from $(0,1)$ onto $(1/2,1)$.
\end{enumerate}
\end{lemma}

\begin{lemma}\label{lem:funct-g}
 For $0<r<1$, let
$$g(r)=\frac{\K\K'(\E\E'+r^2\K\K'-\K\E')}{(4\E'\K-\pi)^2}.$$
 Then $g(r)>0$ for all $r\in(0,1)$, and $g(0^+)=g(1^-)=0$ .
\end{lemma}

\begin{proof} 
By the formula (\ref{eqn:E}) and the parts (1) and (2) of Lemma \ref{lem:citedlem1}, we have
$$\frac{(4\E'\K-\pi)^2}{\K\K'}g(r)=\E\left(\E'-\frac{r'^2\K}{\E}\,\frac{\E'-r^2\K'}{r'^2}\right)>0,$$
and hence $g(r)>0$ for all $r\in(0,1)$.

By Lemma \ref{lem:citedlem1} (3) and (4), we get
\begin{eqnarray*}
g(0^+)&=&\lim\limits_{r\rightarrow 0+}g(r)=\lim\limits_{r\rightarrow 1-}\frac{\K\K'(\E\E'+r'^2\K\K'-\K'\E)}{(4\E\K'-\pi)^2}\\
&=&\lim\limits_{r\rightarrow 1-}\frac{\K'}{(4\E\K'-\pi)^2}\left[(r'\K)^2\K'-\frac{\K'-\E'}{r'^2\K'}(r'^2\K)(\K'\E)\right]\\
&=& 0,
\end{eqnarray*}
and
$$g(1^-)=\lim\limits_{r\rightarrow 1-}g(r)=\lim\limits_{r\rightarrow 1-}\frac{\K'}{(4\E'-\pi/{\K})^2}\left(\frac{\E\E'}{\K}+r^2\K'-\E'\right)=0.$$
\end{proof}

\begin{lemma}\label{lem:funct-h}
Let $\lambda$ be a real number. The
function 
$$h(r)=\frac{4\E'\K-\pi}{m(r)}\left(\frac{m(r)}{r}\right)^{\lambda}$$
 is strictly increasing on
$(0,1)$ if and only if $\lambda\leq 0$.
\end{lemma}

\begin{proof}
By logarithmic differentiation,
\begin{eqnarray}\label{eqn:funct-h}
\frac{h'(r)}{h(r)}&=&\frac{4\left[\K\frac{\E'-\K'}{r'}(-\frac{r}{r'})+\E'\frac{\E-r'^2\K}{rr'^2}\right]}{4\E'\K-\pi}-\frac{\pi-4\E'\K}{{\pi} rm(r)}+\lambda\left(\frac{\pi-4\E'\K}{{\pi}rm(r)}-\frac{1}{r}\right)\nonumber\\
&=&\left(\frac{4\E'\K-\pi}{\pi r m(r)}+\frac1r\right)\times\nonumber\\
& &\left[\frac{1}{1+\frac{{\pi}m(r)}{4\E'\K-\pi}}\left(1+8\frac{\K\K'(\E\E'+r^2\K\K'-\K\E')}{(4\E'\K-\pi)^2}\right)-\lambda\right], \end{eqnarray}
which is positive for all $r\in(0,1)$ if and only if $\lambda\leq0$
by Lemma \ref{lem:funct-g}, since ${m(r)}/{(4\E'\K-\pi)}$ is clearly decreasing from $(0,1)$ onto $(0,\infty)$.
\end{proof}

\begin{remark}\label{rmk:const-C}
Let 
$$H(r)=\frac{1}{1+\frac{{\pi}m(r)}{4\E'\K-\pi}}\left(1+8\frac{\K\K'(\E\E'+r^2\K\K'-\K\E')}{(4\E'\K-\pi)^2}\right)$$ 
for $r\in(0,1)$ and $H(0)=0$, $H(1)=1$. Then $H$ is continuous on $[0,1]$. Hence there exists $r_0\in(0,1]$ such that $H$ obtains its maximum $C$ at $r_0$. Thus it is easy to conclude from (\ref{eqn:funct-h}) that $h$ is strictly decreasing on $(0,1)$ if and only if $\lambda\geq C$.
\end{remark}

\begin{openproblem}
What is the exact expression for $C$?
\end{openproblem}

\section{Proof of the main result}

We are now in a position to prove the main result.

\begin{proof}[Proof of Theorem \ref{thm:pm4m}] 
We first prove the inequality (\ref{ineq:pm4m}) for $\lambda\neq0$. We may assume that $x\leq y$.
Define
$$F(x,y)=m\left(M_\lambda(x,y)\right)^\lambda-\frac{m(x)^\lambda+m(y)^\lambda}{2},\quad \lambda\neq0.$$
Let $t=M_\lambda(x,y)$, then $\frac{\partial t}{\partial x}=\frac12(\frac xt)^{\lambda-1}$. If $x<y$, we have that $t>x$ .
By differentiation,
\begin{eqnarray*}
\frac{\partial F}{\partial x}&=&\frac{\lambda}2m(t)^{\lambda-1}\frac{\pi-4\E'(t)\K(t)}{\pi t}\left(\frac{x}{t}\right)^{\lambda-1}
    -\frac{\lambda}2m(x)^{\lambda-1}\frac{\pi-4\E'(x)\K(x)}{\pi x}\\
    &=&\frac{\lambda x^{\lambda-1}}{2\pi}\left[\frac{4\E'(x)\K(x)-\pi}{m(x)}\left(\frac{m(x)}{x}\right)^{\lambda}
    -\frac{4\E'(t)\K(t)-\pi}{m(t)}\left(\frac{m(t)}{t}\right)^{\lambda}\right]
\end{eqnarray*}
which is positive if and only if $\lambda<0$ by Lemma \ref{lem:funct-h}. Hence
$F(x,y)$ is strictly increasing with respect to $x$ and $F(x,y)\leq
F(y,y)=0$. We now obtain the inequality
$$m\left(M_\lambda(x,y)\right)^\lambda\leq\frac{m(x)^\lambda+m(y)^\lambda}{2},$$
that is $m\left(M_\lambda(x,y)\right)\geq
M_{\lambda}\left(m(x),m(y)\right)$ if and only if $\lambda<0$, with the
equality if and only if $x=y$.

Similarly, by the statement in Remark \ref{rmk:const-C} one can see that the reverse of (\ref{ineq:pm4m}) holds for all $x,y\in (0,1)$ if and only if
${\lambda}\geq C>0$, $C$ is the same as Remark \ref{rmk:const-C}, with the
equality if and only if $x=y$.

Now we prove the inequality
(\ref{ineq:pm4m}) for $\lambda=0$. We may assume that $x\leq y$.
Define
$$G(x,y)=\frac{m(\sqrt{xy})^2}{m(x)m(y)}.$$
Let $t=\sqrt{xy}$, then $\frac{\partial t}{\partial x}=\frac12\frac tx$. If $x<y$, we have that $t>x$ .
By logarithmic differentiation, we have
\begin{eqnarray*}
\frac1{G(x,y)}\frac{\partial G}{\partial x}&=&\frac1{m(t)}\frac{\pi-4\E'(t)\K(t)}{\pi t}\frac{t}{x}
    -\frac1{m(x)}\frac{\pi-4\E'(x)\K(x)}{\pi x}\\
    &=&\frac1{\pi x}\left(\frac{4\E'(x)\K(x)-\pi}{m(x)}-\frac{4\E'(t)\K(t)-\pi}{m(t)}\right)
\end{eqnarray*}
which is negative. Hence
$G(x,y)$ is strictly decreasing with respect to $x$ and $G(x,y)\geq
G(y,y)=1$. We now obtain the inequality
$$m(\sqrt{xy})\geq\sqrt{m(x)m(y)},$$
that is $m\left(M_0(x,y)\right)\geq
M_0\left(m(x),m(y)\right)$, with the
equality if and only if $x=y$. This completes the proof. 
\end{proof}

The following corollary is clear.

\begin{corollary} 
The function $m(r)$ is concave on $(0,1)$ with respect to power mean of order $\lambda$ if and only if $\lambda\leq0$,
and convex on $(0,1)$ with respect to power mean of order $\lambda$ if and only if $\lambda\geq C$.
\end{corollary}

\medskip

\subsection*{Acknowledgments}

This research is partly supported by the NNSF of China (No. 11071069) and the NSF of Zhejiang Province (No. Y6100170).

\end{document}